\documentclass[12pt]{amsart}

\usepackage{amscd,latexsym,amsthm,amsfonts,amssymb,amsmath,amsxtra}
\usepackage[mathscr]{eucal}
\usepackage{hyperref}
\usepackage{pdfsync}
\renewcommand\theequation{\thesection.\arabic{equation}}

\makeatletter
\def\Ddots{\mathinner{\mkern1mu\raise\p@
\vbox{\kern7\p@\hbox{.}}\mkern2mu
\raise4\p@\hbox{.}\mkern2mu\raise7\p@\hbox{.}\mkern1mu}}
\makeatother

\makeatletter
\def\Ddots{\mathinner{\mkern1mu\raise\p@
\vbox{\kern7\p@\hbox{.}}\mkern2mu
\raise4\p@\hbox{.}\mkern2mu\raise7\p@\hbox{.}\mkern1mu}}
\makeatother

\newcommand{\GL}{{\mathrm{GL}}}

\newcommand{\Ind}{{\mathrm{Ind}}}

\newtheorem{thm}{Theorem}[section]
\newtheorem{cor}[thm]{Corollary}
\newtheorem{lem}[thm]{Lemma}

\newtheorem {ques/conj}[thm]{Question/Conjecture}

\newtheorem{rmk}[thm]{Remark}

\makeatletter

\newcommand{\Rmnum}[1]{\expandafter\@slowromancap\romannumeral #1@}
\makeatother

\begin{document}
\renewcommand{\theequation}{\arabic{equation}}
\numberwithin{equation}{section}

\title[Langlands Parameter for Simple Supercuspidals]{On The Langlands parameter of a Simple Supercuspidal Representation: Odd Orthogonal Groups}

\begin{abstract}
In this work, we explicitly compute a certain family of twisted gamma factors of a simple supercuspidal representation $\pi$ of a $p$-adic odd orthogonal group.  These computations, together with analogous computations for general linear groups carried out in previous work with Liu \cite{AL14}, allow us to give a prediction for the Langlands parameter of $\pi$.  If we assume the ``depth-preserving conjecture'', we prove that our prediction is correct if $p$ is sufficiently large.
\end{abstract}

\author{Moshe Adrian}
\address{Department of Mathematics\\
University of Toronto\\
Toronto, ON Canada M5S 2E4}
\email{madrian@math.toronto.edu}

\date{\today}


\thanks{MA is is supported in part by postdoc research funding from the Department of Mathematics at University of Toronto.}
\maketitle

\tableofcontents

\section{Introduction}\label{introduction}
Let $\mathbf{G}$ be a connected reductive group defined over a $p$-adic field $F$.  Recently, Gross, Reeder, and Yu \cite{GR10, RY13} have constructed a class of supercuspidal representations of $G = \mathbf{G}(F)$, called \emph{simple supercuspidal representations}.  These are the supercuspidal representations of $G$ of minimal positive depth.

Of interest is the Langlands correspondence for simple supercuspidal representations of $G$.  Recall that the Langlands correspondence (or LLC) is a certain conjectural finite-to-one map $$\Pi(G) \rightarrow \Phi(G)$$ from equivalence classes of irreducible admissible representations of $G$ to equivalence classes of  Langlands parameters of $G$.  In the past several years there has been much activity on the Langlands correspondence.  DeBacker and Reeder \cite{DR09} have constructed a correspondence for depth zero supercuspidal representations of unramified $p$-adic groups, and Reeder \cite{R08} extended this construction to certain positive depth supercuspidal representations of unramified groups.  Recently, under a mild assumption on the residual characteristic, Kaletha \cite{K13} has constructed a correspondence for simple supercuspidal representations of $p$-adic groups, and has verified that his correspondence satisfies many of the expected properties of the Langlands correspondence.  Kaletha has subsequently \cite{K15} extended these results to epipelagic supercuspidal representations of $p$-adic groups.

Our present work is the first in a series of papers dedicated to explicitly determining the Langlands parameter $\varphi_{\pi}$ of a simple supercuspidal representation $\pi$ of a classical group $G$, using the theory of gamma factors.  Let $\pi$ and $\tau$ be a pair of irreducible representations of $G$ and $GL_n$, where $G$ is either $SO_{2\ell}$, $SO_{2\ell+1}$ or $Sp_{2\ell}$.  Here, we assume that $SO_{2\ell+1}$ is split and $SO_{2\ell}$ is quasi-split.  Fix a nontrivial character $\psi$ of $F$.  The Rankin-Selberg integral for $G \times GL_n$ was constructed, in different settings, in a series of works including \cite{GPSR87, G90, GPSR97, GRS98, Sou93}.  If $G$ is orthogonal, we denote this Rankin-Selberg integral by $\Phi(W,f_s)$, and $\Phi(W,\phi, f_s)$ otherwise.  Here, $W$ is a Whittaker function for $\pi$, $f_s$ is a certain holomorphic section of a principal series (induced from $\tau$) depending on the complex parameter $s$, and $\phi$ is a Schwartz function.  Applying a standard normalized intertwining operator $M^*(\tau,s)$ to $f_s$, one obtains a similar integral $\Phi^*(W,f_s) = \Phi(W, M^*(\tau,s) f_s)$ if $G$ is orthogonal (and $\Phi^*(W, \phi,f_s) = \Phi(W, \phi, M^*(\tau,s) f_s)$ otherwise) related to $\Phi(W,f_s)$ (or $\Phi(W,\phi,f_s)$) by a functional equation

\[
\left\{
\begin{array}{lll}
\gamma(s, \pi \times \tau, \psi) \Phi(W,f_s) = \Phi^*(W,f_s) & \text{if} & G = SO_{2 \ell + 1}\\
\gamma(s, \pi \times \tau, \psi) \Phi(W,f_s) = c(s, \ell, \tau, \gamma) \Phi^*(W,f_s) & \text{if} & G = SO_{2 \ell}\\
\gamma(s, \pi \times \tau, \psi) \Phi(W,\phi, f_s) = \Phi^*(W,\phi, f_s) & \text{if} & G = Sp_{2 \ell}
\end{array} \right.
\]

For an accessible paper discussing all of these functional equations, we refer the reader to the recent work of Kaplan \cite{K14ii}.

The term $\gamma(s, \pi \times \tau, \psi)$ is known as the \emph{gamma factor} of $\pi \times \tau$, and is the key ingredient in determining $\varphi_{\pi}$ in this paper.  To determine $\varphi_{\pi}$, one must locate the poles of the  gamma factors $\gamma(s, \pi \times \tau, \psi)$ where $\tau$ ranges over the supercuspidal representations of various general linear groups.  The location of poles will determine the isobaric constituents of a specific functorial lift $\Pi$ to a general linear group $GL$, the lift corresponding to the standard $L$-homomorphism ${}^L G \rightarrow {}^L GL$ (see \cite[\S2]{ACS14}).  In the case of simple supercuspidal representations $\pi$ of a classical group $G$, it is expected that the isobaric constituents of $\Pi$ are either tamely ramified characters of $GL_1$, or simple supercuspidal representations (for $p$ large, this can be seen from the so-called ``depth-preserving conjecture''; see \S\ref{depthpreservation}).  We can then use the explicit local Langlands correspondence for simple supercuspidal representations of general linear groups (see \cite{AL14, BH10, BH13}) to determine $\varphi_{\pi}$.

While describing the Langlands correspondence explicitly is in general a difficult task, in the simple supercuspidal setting it turns out to be tractable.  To explain our approach, we first make some observations.  If $\pi$ is a simple supercuspidal representation of a classical group, it is expected (as remarked earlier) that its Langlands parameter $\varphi_{\pi}$ decomposes as an irreducible subrepresentation $\varphi_{\pi,1}$ (that corresponds to a simple supercuspidal representation of a general linear group under LLC) and a (possibly empty) direct sum $\varphi_{\pi,2}$ of tamely ramified one-dimensional subrepresentations.  Using this prediction as our guide, we can describe our method to fully determine $\varphi_{\pi}$ in two steps.  The first is to compute the gamma factor $\gamma(s, \pi \times \tau, \psi)$, where $\tau$ is any tamely ramified character of $GL_1$.  This computation will determine $\varphi_{\pi,2}$.  By multiplicativity of gamma factors, it will also yield a prediction for a specific simple supercuspidal representation $\sigma$ (whose associated Langlands parameter is $\varphi_{\pi,1}$) that is an isobaric constituent of $\Pi$, leading to step two.  In step two, one computes $\gamma(s, \pi \times \sigma, \psi)$ in order to prove that $\sigma$ is indeed such a constituent.  We can then use LLC for simple supercuspidal representations of general linear groups to determine $\varphi_{\pi,1}$.

We now give more details on these steps in the case of $G = SO_{2\ell+1}$, which is the central concern of this paper.  The simple supercuspidal representations of $G$ are parameterized by two pieces of data: a choice of a uniformizer $\varpi$ in $F$, and a sign.  More explicitly (see \S\ref{simplesupercuspidalsoddorthogonal}), let $\chi$ be an affine generic character of the pro-unipotent radical $I^+$ of an Iwahori $I$.  The choice of a uniformizer $\varpi$ in $F$ determines an element $g_{\chi}$ in $G$ which normalizes $I$ and stabilizes $\chi$.  We can extend $\chi$ to $K = \langle g_{\chi} \rangle I^+$ in two different ways, since $g_{\chi}^2 = 1$, and $\pi = \mathrm{Ind}_K^G \chi$ is simple supercuspidal.

Let $\tau$ be a tamely ramified character of $GL_1$.  Our main result (Theorem \ref{gammafactor}) is

\begin{thm}\label{maintheorem}
\[
\gamma(s, \pi \times \tau, \psi) = \chi(g_{\chi}) \tau(-\varpi) q^{1/2 - s}.
\]
\end{thm}

In particular, since there are no poles in $\gamma(s, \pi \times \tau, \psi)$, there can be no tamely ramified character of the Weil group $W_F$ appearing as a summand of $\varphi_{\pi}$.  In particular, it is therefore expected that $\varphi_{\pi}$ is irreducible. Therefore, on the one hand, the computation of $\gamma(s, \pi \times \tau, \psi)$ determines no part of $\varphi_{\pi}$.  On the other hand, it does tell us precisely what the Langlands parameter should be, as follows.

By the results of \cite{AL14, BH13} (see in particular \cite[Remark 3.18]{AL14}), there exists a unique irreducible $2\ell$-dimensional representation $\varphi : W_F \rightarrow GL(2\ell,\mathbb{C})$ with trivial determinant, whose gamma factor $\gamma(s, \varphi \times \tau, \psi)$ equals $\chi(g_{\chi}) \tau(-\varpi) q^{1/2 - s}$ for every tamely ramified character $\tau$ (we give more details on this in \S\ref{whittakersection}).  The Langlands parameter $\varphi$ has  been explicitly described in \cite{AL14, BH10, BH13}.  In the case that $p \nmid 2\ell$, let us briefly recall this description (here we follow \cite{AL14}).

Let $\varpi$ be a uniformizer of $F$, let $\varpi_E$ be a $2\ell^{\mathrm{th}}$ root of $\varpi$, and set $E = F(\varpi_E)$.  Relative to the basis
$$\varpi_E^{2\ell-1}, \varpi_E^{2\ell-2}, \cdots, \varpi_E, 1$$
of $E/F$, we have an embedding $$\iota : E^{\times} \hookrightarrow GL_{2\ell}.$$ Define a character $\xi$ of $E^{\times}$ by setting $\xi|_{1 + \mathfrak{p}_E} = \lambda \circ \iota$, where $\lambda(A) := \psi(A_{12} + A_{23} + ... + A_{2\ell-1,2\ell} + \frac{1}{\varpi} A_{2\ell,1})$, for $A = (A)_{ij} \in GL_{2\ell}$.  Moreover, we define $\xi|_{k_F^{\times}} \equiv (\varkappa_{E/F}|_{k_F^{\times}})^{-1}$, where $\varkappa_{E/F} = \mathrm{det}(\Ind_{W_E}^{W_F}(1_E))$, $1_E$ denotes the trivial character of $W_E$, and $k_F$ is the residue field of $F$.  Finally, we set $$\xi(\varpi_E) = \chi(g_{\chi}) \lambda_{E/F}(\psi)^{-1},$$ where $\lambda_{E/F}(\psi)$ is the Langlands constant (see \cite[\S34.3]{BH06}).  Then $\varphi = \mathrm{Ind}_{W_E}^{W_F} \xi$.  The proof that $\varphi_{\pi} = \varphi$ boils down to the computation of $\gamma(s, \pi \times \Pi, \psi)$, where $\Pi$ is the simple supercuspidal representation of $GL_{2\ell}$ corresponding to $\varphi$ under LLC, and will be carried out in future work. This is our ``step two'' described above.  On the other hand, if we assume the ``depth-preserving conjecture'', we can easily prove (see Corollary \ref{functoriallift}) that $\varphi_{\pi} = \varphi$ for $p$ sufficiently large.

The cases of $SO_{2\ell}$ and $Sp_{2\ell}$ are in some sense similar.  For example, suppose that $G = SO_{2\ell}$.  Let $\pi$ be a simple supercuspidal representation of $G$, and let $\varphi_{\pi}$ be its Langlands parameter.  Again, it is expected that $\varphi_{\pi}$ decomposes as an irreducible subrepresentation $\varphi_{\pi,1}$ (that corresponds to a simple supercuspidal representation under LLC) and a (possibly empty) direct sum $\varphi_{\pi,2}$ of tamely ramified one-dimensional subrepresentations (again, for $p$ large, this can be seen from the ``depth-preserving conjecture'').  It is known that all summands must be self-dual.  Therefore the computation of $\gamma(s, \pi \times \tau, \psi)$, where $\tau$ is a quadratic and tamely ramified character of $GL_1$ will determine $\varphi_{\pi,2}$.  The computation of $\gamma(s, \pi \times \tau, \psi)$, where $\tau$ is tamely ramified and not quadratic, will give a specific prediction (as in the case of $SO_{2\ell+1}$) for $\varphi_{\pi,1}$.  Then one must perform ``step two'' by calculating $\gamma(s, \pi \times \Pi, \psi)$, where $\Pi$ is the unique simple supercuspidal representation (with trivial central character) of a certain general linear group such that $\gamma(s, \Pi \times \tau, \psi) = \gamma(s, \pi \times \tau, \psi)$ for all tamely ramified characters $\tau$ of $GL_1$.

We now summarize the contents of the sections of this paper.  In \S\ref{preliminaries}, we recall the functional equation for odd orthogonal groups as in \cite{Sou93}.  In \S\ref{whittakersection}, we recall a construction of simple supercuspidal representations of $GL_n$ as well as their tamely ramified $GL_1$-twisted gamma factors.  In \S\ref{simplesupercuspidalsoddorthogonal}, we give a construction of simple supercuspidal representations of odd orthogonal groups that may be used in computing the Rankin-Selberg integrals from \S\ref{preliminaries}. In \S\ref{easyside} and \S\ref{hardside}, we compute a family of Rankin-Selberg integrals that allow us to explicitly compute the tamely ramified $GL_1$-twisted gamma factors of a simple supercuspidal representation of $SO_{2\ell+1}$.  The main result on the values of these gamma factors is Theorem \ref{gammafactor}, which allows us to predict the Langlands parameter of a simple supercuspidal representation of an odd orthogonal group.  Finally, in \S\ref{depthpreservation}, we show that if we assume a conjecture on depth preservation, then our prediction for the Langlands parameter is indeed the correct Langlands parameter.

\subsection*{Acknowledgements}
This paper has benefited from conversations with Benedict Gross, Tasho Kaletha, and Eyal Kaplan.  We thank them all.

\section{Notation}\label{notation}

Let $F$ be a $p$-adic field, with ring of integers $\mathfrak{o}$, maximal ideal $\mathfrak{p}$, $k_F = \frak{o} / \frak{p}$, let $q$ denote the cardinality of $k_F$, and fix a uniformizer $\varpi$.
We normalize measures $dv$ and $d^{\times} v$ on $F$ and $F^{\times}$ by taking $vol(\frak{o}) = q^{\frac{1}{2}}$ and
$vol(\frak{o}^{\times}) = 1$.  Throughout, we fix a nontrivial additive character $\psi : F \rightarrow \mathbb{C}^{\times}$ that is trivial on $\mathfrak{p}$ but nontrivial on $\mathfrak{o}$.  Given a connected reductive group $\mathbf{G}$ defined over $F$, let $G=\mathbf{G}(F)$.  We also let $W_F$ denote the Weil group of $F$.

\section{The local functional equation for odd orthogonal groups}\label{preliminaries}
In this section we recall the functional equation for odd orthogonal groups, as in \cite[\S1.3, \S9.6, \S10.1]{Sou93}.\footnote{There is actually a typo in \cite{Sou93} as observed in \cite[Remark 3.1]{K14ii}.  However, this does not affect our present work (see Remark \ref{soudrytypo}).}

Let $J_\ell$ denote the $\ell \times \ell$ matrix
$\begin{pmatrix}
 & & 1\\
 & \Ddots & \\
 1 & &
\end{pmatrix}$.  We define
\begin{equation}\label{SOodddefinitions}
SO_{2\ell+1} = \left\{ g \in GL_{2\ell+1} : det(g) = 1, {}^t g
J_{2\ell+1}
g =
J_{2\ell+1}
\right\},
\end{equation}
\begin{equation}\label{SOevendefinition}
SO_{2\ell} =
\left\{ g \in GL_{2\ell} : det(g) = 1, {}^t g
J_{2\ell}
g =
J_{2\ell}
\right\}.
\end{equation}
Let $T_{SO_{2\ell}}$ be the split maximal torus of $SO_{2\ell}$.  Let $\Delta_{SO_{2\ell}}$ be the standard set of simple roots of $SO_{2\ell}$, so that $\Delta_{SO_{2\ell}} = \{ \epsilon_1 - \epsilon_2, ..., \epsilon_{\ell-1} - \epsilon_\ell,  \epsilon_{\ell-1} + \epsilon_\ell \}$, where $\epsilon_i(t) = t_i$ is the $i$-th coordinate function of $t \in T_{SO_{2\ell}}$.  Let $Q_{\ell}$ be the standard maximal parabolic subgroup of $SO_{2\ell}$ corresponding to $\Delta \setminus \{  \epsilon_{\ell-1} + \epsilon_\ell \}$.  The Levi part of $Q_{\ell}$ is isomorphic to $GL_\ell$, and we denote its unipotent radical by $U_{\ell}$.

Let $T_{SO_{2\ell+1}}$ be the split maximal torus of $SO_{2\ell+1}$.  let $\Delta_{SO_{2\ell+1}}$ be the standard set of simple roots of $SO_{2\ell+1}$, so that $\Delta_{SO_{2\ell+1}} = \{ \epsilon_1 - \epsilon_2, ..., \epsilon_{\ell-1} - \epsilon_\ell,  \epsilon_\ell \}$.  The highest root is $\epsilon_1 + \epsilon_2$.  Let $U_{SO_{2\ell+1}}$ denote the subgroup of upper triangular unipotent matrices.

Let $\psi$ be a nontrivial additive character of $F$.
We let $U_{GL_\ell}$ denote the maximal unipotent subgroup of $GL_\ell$ consisting of upper triangular matrices.  We also denote by $\psi$ the standard nondegenerate Whittaker character of
\[
\psi(u) = \psi \left( \displaystyle\sum_{i=1}^{\ell-1} u_{i,i+1} \right),
\]
for $u = (u_{ij}) \in U_{GL_\ell}$.

Let $\tau$ be a generic representation of $GL_\ell$.  We denote by $\mathcal{W}(\tau, \psi)$ the Whittaker model of $\tau$ with respect to $\psi$.

For $s \in \mathbb{C}$, define
\[
V_{Q_\ell}^{SO_{2\ell}}(\mathcal{W}(\tau, \psi),s) = \mathrm{Ind}_{Q_\ell}^{{SO}_{2\ell}} (\mathcal{W}(\tau, \psi) |\mathrm{det}|^{s-1/2} ).
\]
Thus, a function $f_{s}$ in $V_{Q_\ell}^{SO_{2\ell}}(\mathcal{W}(\tau, \psi),s)$ is a smooth function on $SO_{2\ell} \times GL_\ell$, where for any $g \in SO_{2\ell}$, $m \in Q_\ell$, and $u$ in the unipotent radical $U_{\ell}$, we have

\begin{equation}\label{sectiondefinition}
f_s((mug, a) = |\mathrm{det}(m)|^{s+\frac{n-2}{2}} f_s(g,am),
\end{equation}
and the mapping $a \mapsto f_s(g,a)$ belongs to $\mathcal{W}(\tau, \psi)$.

We now let $\ell,n \in \mathbb{N}$ such that $\ell \geq n$.  Let $U_{SO_{2\ell+1}}$ denote the subgroup of upper triangular unipotent matrices in $SO_{2\ell+1}$.  We use the notation $\psi$ again to define a non-degenerate character on $U_{SO_{2\ell+1}}$ by
\[
\psi(u) = \psi \left( \displaystyle\sum_{i=1}^{\ell} u_{i,i+1} \right).
\]

Let $\pi$ be a generic representation of $SO_{2\ell+1}$ with respect to $\psi$.  We shall denote by $\mathcal{W}(\pi, \psi)$ the Whittaker model of $\pi$ with respect to $\psi$.  Let $\tau$ be a generic representation of $GL_n$ with respect to $\psi$.

In computing $\gamma$-factors, we will be interested in the following local integrals (see \cite{Sou93}):
\begin{align*}\label{zeta}
\begin{split}
 & \ \Phi(W, f_s)\\
: = & \ \int_{U_{SO_{2n}} \setminus SO_{2n}} \int_{\overline{X}_{(n,\ell)}}
W(\overline{x} j_{n,\ell}(h)) f_s(h, I_n) d \overline{x} dh,
\end{split}
\end{align*}
where $W \in \mathcal{W}(\pi, \psi)$, $f_s \in V_{Q_n}^{SO_{2n}}(\mathcal{W}(\tau, \psi^{-1}),s)$, and where $dh$ is a fixed right-invariant Haar measure on the quotient space $U_{SO_{2n}} \setminus SO_{2n}$, and $d \overline{x}$ is the product measure inherited from $dv$.

Here, $j_{n,\ell} : SO_{2n} \rightarrow SO_{2\ell+1}$ is the map
\[
\begin{pmatrix}
A & B\\
C & D
\end{pmatrix} \mapsto
\begin{pmatrix}
A & & B\\
& I_{2(\ell-n)+1} & \\
C & & D
\end{pmatrix} \in SO_{2\ell+1},
\]
\[
\overline{X}_{(n,\ell)} =
\begin{pmatrix}
I_n & & & & \\
y & I_{\ell-n} & & & \\
& & 1 & &\\
& & & I_{\ell-n} &\\
& & & y' & I_n
\end{pmatrix} \in SO_{2\ell+1}.
\]

Note that $y \in F^{\ell-1}$ is a column vector, and $y' \in F^{\ell-1}$ is a row vector.  Moreover, the notation $y'$ means that these coordinates are determined by the coordinates of $y$, according to the matrix defining $SO_{2\ell+1}$.

Now let $n$ be odd.  We define an intertwining operator by

\[
M(\tau,s) f_s(h,a) = \int_{\overline{U_n}} f_s(\overline{u} w_n^{-1} h, a) d \overline{u},
\]
where
\[
w_n =
\begin{pmatrix}
& I_n \\
I_n &
\end{pmatrix}
\begin{pmatrix}
& & 1\\
& I_{2(n-1)} &\\
1 & &
\end{pmatrix} \in SO_{2n},
\]
and where $\overline{U_n}$ is the opposite unipotent radical to $U_n$.

Then we will also be interested in the integrals (see \cite[\S9.6]{Sou93}) $\Phi^*(W, f_s)$, where
\begin{align*}
\begin{split}
 & \ \Phi^*(W, f_s)\\
: = & \ \gamma(2s-1, \tau, \wedge^2, \psi) \int_{U_{SO_{2n}} \setminus SO_{2n}} \int_{\overline{X}_{(n,\ell)}}
W(\hat{c}_{n,\ell} \overline{x} j_{n,\ell}(h) \delta_o \omega') M(\tau,s) f_s({}^{\omega} h, b_n^*) d \overline{x} dh.
\end{split}
\end{align*}
Here, if $g \in GL_n$, then $g^* = J_n {}^{t} g^{-1} J_n$.  Moreover,
\begin{align*}
\begin{split}
& \ \hat{c}_{n,\ell} = diag(I_n, -I_{\ell-n},1,-I_{\ell-n},I_n) \in SO_{2\ell+1},\\
& \ \delta_o = diag(I_\ell,-1,I_\ell),\\
& \ b_n = diag(1,-1,...,-1,1) \in GL_n,
\end{split}
\end{align*}
\[
\omega' =
\begin{pmatrix}
I_{n-1} & & & &\\
& & & 1 &\\
& & I_{2(\ell-n)+1} & &\\
& 1 & & & \\
& & & & I_{n-1}
\end{pmatrix},
\]
\[
\omega =
\begin{pmatrix}
I_{n-1} & & &\\
& 0 & 1 &\\
& 1 & 0 &\\
& & & I_{n-1}
\end{pmatrix}.
\]
We then have the functional equation for odd orthogonal groups.
\begin{thm}\cite[\S10.1]{Sou93}
\begin{equation}
\gamma(s, \pi \times \tau, \psi) \Phi(W,f_s) = \Phi^*(W,f_s). \label{functionalequation}
\end{equation}
\end{thm}

\begin{rmk}\label{soudrytypo}
In \cite[Remark 3.1]{K14ii} it was observed that there is a typo in \cite{Sou93} having to do with the Whittaker space for $\pi$.  However, as noted in the same remark, this typo only needs to be fixed in the case that $l < n$.  Since we are at present only considering $n = 1$, we may still follow the notation and integrals from \cite{Sou93}.
\end{rmk}

\section{The simple supercuspidal representations of $GL_n$}\label{whittakersection}
In this section we review the definition of a simple supercuspidal representation of $GL_n$, as in \cite{AL14} (see also \cite{KL15}). We then recall the values of a family of its twisted gamma factors, in order to predict the Langlands parameter of a simple supercuspidal of $SO_{2\ell+1}$.

Let $GL_n = \mathbf{GL_n}(F)$, $Z$ the center of $G$, $K = \mathbf{GL_n}(\mathfrak{o})$ the standard maximal compact subgroup, $I$ the standard Iwahori subgroup, and $I^+$ its pro-unipotent radical.  Fix a nontrivial additive character $\psi$ of $F$ that is trivial on $\mathfrak{p}$ and nontrivial on $\mathfrak{o}$.  Fix a uniformizer $\varpi$ in $F$.  Let $U$ denote the unipotent radical of the standard upper triangular Borel subgroup of $GL_n$. For any $u \in U$, we recall from \S\ref{preliminaries} the standard nondegenerate Whittaker character $\psi$ of $U$:
\[
\psi(u) = \psi \left( \displaystyle\sum_{i = 1}^{n-1} u_{i,i+1} \right).
\]
Set $H = ZI^+$, and fix a character $\omega$ of $Z \cong F^{\times}$, trivial on $1 + \mathfrak{p}$.  For $t \in \mathfrak{o}^{\times} / (1 + \mathfrak{p})$, we define a character $\chi : H \rightarrow \mathbb{C}^{\times}$ by $\chi(zk) = \omega(z) \psi(r_1 + r_2 ... + r_{n-1} + r_n)$ for
\[
k = \left( \begin{array}{ccccc}
x_1 & r_1 & * &  \cdots & \\
* & x_2 & r_2 & \cdots & \\
\vdots & & \ddots & \ddots & \\
* & &  &  & r_{n-1}\\
\varpi r_n &  & \cdots &  & x_n
\end{array} \right)
\]

The compactly induced representation $\pi_{\chi} := \mathrm{cInd}_H^G \chi$ is then a direct sum of $n$ distinct irreducible supercuspidal representations of $GL_n$. They are parameterized by $\zeta$, where $\zeta$ is a complex $n^{\mathrm{th}}$ root of $\omega( \varpi)$, as follows.  Set

\[
g_{\chi} = \left( \begin{array}{ccccc}
0 & 1 &  &   & \\
 &  & 1 &  & \\
 & & & \ddots & \\
 & &  &  & 1\\
\varpi  &  &  &  & 0
\end{array} \right)
\]
and set $H' = \langle g_{\chi} \rangle H$.  Then the summands of $\pi_{\chi}$ are $$\sigma_{\chi}^{\zeta} := \mathrm{cInd}_{H'}^G \chi_{\zeta}$$ where $\chi_{\zeta}(g_{\chi}^j h) = \zeta^j \chi(h)$, as $\zeta$ runs over the complex $n^{\mathrm{th}}$ roots of $\omega(\varpi)$.  The $\sigma_{\chi}^{\zeta}$ are called the \emph{simple supercuspidal} representations of $GL_n$.

Finally, for the simple supercuspidal representation $\sigma_{\chi}^{\zeta}$, we may define a Whittaker function on $GL_n$ by setting

\begin{equation*}
W(g) = \left\{
\begin{array}{rll}
\psi(u) \chi_{\zeta}(h') & \text{if} & g = uh' \in U H'\\
0 &  & \text{else}
\end{array} \right.
\end{equation*}
This function is well-defined, by definition of $\psi$ and $\chi_{\zeta}$.  Using this Whittaker function, it turns out to be not so difficult to compute
the Rankin-Selberg integrals that arise in the definition of the twisted gamma factors $\gamma(s, \sigma_{\chi}^{\zeta} \times \tau, \psi)$, where $\tau$ is a tamely ramified character of $F^{\times}$.  In \S\ref{simplesupercuspidalsoddorthogonal}, we will define an analogous Whittaker function for a simple supercuspidal representation $\pi$ of $SO_{2\ell+1}$, which will allow us to compute the relevant Rankin-Selberg integrals for $\pi$.

In \cite{AL14}, we explicitly computed $\gamma(s, \sigma_{\chi}^{\zeta} \times \tau, \psi)$:

\begin{lem}\cite[Lemma 3.14]{AL14}\label{simplesupercuspidalgammagln}
Let $\tau$ be a tamely ramified character of $GL_1$.  Then
\[
\gamma(s, \sigma_{\chi}^{\zeta} \times \tau, \psi) = \tau(-1)^{n-1} \tau(\varpi) \chi(g_{\chi}) q^{1/2 - s}
\]
\end{lem}

This result allows us to conclude (see \cite[Remark 3.18]{AL14}) that a simple supercuspidal of $GL_n$, up to central character, can be distinguished from any other supercuspidal of $GL_n$ by the above twisted $\gamma$-factors.

\section{The simple supercuspidal representations of $SO_{2\ell+1}$}\label{simplesupercuspidalsoddorthogonal}

In this section we construct the simple supercuspidal representations of $SO_{2\ell+1}$, and define an explicit Whittaker function for each simple supercuspidal.

Let $T = T_{SO_{2\ell+1}}$ denote the split maximal torus of $SO_{2\ell+1}$. Associated to $T$ we have the set of affine roots $\Psi$. Let $X^*(T)$ denote the character lattice of $T$, let $T_0$ be the maximal compact subgroup of $T_0$, and set $$T_1 = \langle t \in T_0 : \lambda(t) \in 1 + \mathfrak{p} \ \forall \lambda \in X^*(T) \rangle.$$

To each $\alpha \in \Psi$ we have an associated affine root group $U_{\alpha}$ in $G$.   Associated to the standard set of simple roots $\Delta_{SO_{2\ell+1}}$ we have a set of simple affine $\Pi$ and positive affine roots $\Psi^+$.  Set $$I = \langle T_0, U_{\alpha} : \alpha \in \Psi^+ \rangle,$$ $$I^+ = \langle T_1, U_{\alpha} : \alpha \in \Psi^+ \rangle,$$  $$I^{++} = \langle T_1, U_{\alpha} : \alpha \in \Psi^+ \setminus \Pi \rangle.$$
In terms of the Moy-Prasad filtration, if $x$ denotes the barycenter of the fundamental alcove, then $I = G_x = G_{x,0}, I^+ = G_{x+} = G_{x,\frac{1}{2\ell}}$, and $I^{++} = G_{x,\frac{1}{2\ell}+}$.

For $(t_1, t_2, ..., t_{\ell+1}) \in \mathfrak{o}^{\times} / (1 + \mathfrak{p}) \times \mathfrak{o}^{\times} / (1 + \mathfrak{p}) \times \cdots \times \mathfrak{o}^{\times} / (1 + \mathfrak{p})$, we define a character

$$\chi : I^+ \rightarrow \mathbb{C}^{\times}$$ $$h \mapsto \psi(t_1 h_{12} + t_2 h_{23} + \cdots + t_{\ell-1} h_{\ell-1,\ell} + t_\ell x_{\ell,\ell+1} + t_{\ell+1} \frac{h_{2\ell,1}}{\varpi})$$
for $h = (h)_{ij} \in I^+$.  These $\chi$'s are called the \emph{affine generic characters} of $I^+$ (see \cite{GR10, RY13}).  The orbits of affine generic characters are parameterized by the set of elements in $\mathfrak{o}^{\times} / (1 + \mathfrak{p})$, as follows.  $T \cap \mathbf{SO_{2\ell+1}}(\mathfrak{o})$ normalizes $I^+$, so acts on the set of affine generic characters by conjugation.  Every orbit of affine generic characters contains one of the form $(1,1,...,1,t)$, for $t \in \mathfrak{o}^{\times} / (1 + \mathfrak{p})$.  Specifically, the orbit of $(t_1, t_2, ..., t_{\ell+1})$ contains $(1,1, ..., 1, t)$, where $t = \frac{1}{t_1 t_2^2 t_3^2 \cdots t_\ell^2}$.

Instead of viewing the affine generic characters as parameterized by $t \in \mathfrak{o}^{\times} / (1 + \mathfrak{p})$, we will set $t = 1$ and let the affine generic characters be parameterized by the various choices of uniformizer in $F$.  Since we have already fixed an (arbitrary) uniformizer ahead of time in \S\ref{notation}, we have therefore fixed an affine generic character $\chi$.  The compactly induced representation $\pi_{\chi} := \mathrm{ind}_H^G \chi$ is a direct sum of $2$ distinct irreducible supercuspidal representations of $SO_{2\ell+1}$. They are parameterized $\zeta \in \{-1,1 \}$, as follows.  Set

\[ g_{\chi} = \begin{pmatrix}
 & & \varpi^{-1}\\
 & -I_{2\ell-1} & \\
\varpi & &
\end{pmatrix}.
\]

One can compute that $g_{\chi}$ normalizes $I^+$, and that $\chi(g_{\chi} h g_{\chi}^{-1}) = \chi(h) \ \forall h \in I^+$. Set $H' = \langle g_{\chi} \rangle I^+$.  Since $g_{\chi}^2 = 1$, we may extend $\chi$ to a character of $H'$ in two different ways by mapping $g_{\chi}$ to $\zeta \in \{ -1, 1 \}$.  Then the summands of $\pi_{\chi}$ are the compactly induced representations
$$
\pi_{\chi}^{\zeta} := \mathrm{ind}_{H'}^G \chi_{\zeta}
$$
where $\chi_{\zeta}(g_{\chi} h) = \zeta \cdot \chi(h)$, where $\zeta \in \{ -1,1 \}$ and $h \in I^+$.  The $\pi_{\chi}^{\zeta}$'s are the \emph{simple supercuspidal} representations of $SO_{2\ell+1}$.

To compute the Rankin-Selberg integrals, we need a Whittaker function for $\pi_{\chi}^{\zeta}$.  We note that $\chi_{\zeta}(u) = \psi \left( \sum_{i=1}^r u_{i,i+1} \right)$ for $u \in U \cap I^+$, so the function

\begin{equation}\label{whittakerfunction}
W(g) = \left\{
\begin{array}{rll}
\psi(u) \chi_{\zeta}(g_{\chi}^i) \chi_{\zeta}(k) & \text{if} & g = u g_{\chi}^i k \in U_{SO_{2\ell+1}} \langle g_{\chi} \rangle I^+ \\
0 &  & \text{else}
\end{array} \right.
\end{equation}

lives in $Ind_{U}^G \psi$, and is therefore a Whittaker function for $\pi_{\chi}^{\zeta}$ with respect to $\psi$.

\section{The computation of $\Phi(W, f_s)$}\label{easyside}
In this section, we compute one side of the functional equation \eqref{functionalequation} that defines $\gamma(s, \pi \times \tau, \psi)$, where $\pi$ is a simple supercuspidal representation of $SO_{2\ell+1}$, and $\tau$ is a tamely ramified character of $GL_1$.

We first note that since $\tau$ is a character of $GL_1$, $\Phi(W, f_s)$ simplifies to

\[
\int_{U_{SO_{2}} \setminus SO_{2}} \int_{\overline{X}_{(1,\ell)}}
W(\overline{x} j_{1,\ell}(h)) f_s(h, 1) d \overline{x} dh,
\]

Moreover, note that by \eqref{SOevendefinition},
\[
SO_2 = \left\{
\begin{pmatrix}
z & 0 \\
0 & z^{-1}
\end{pmatrix} : z \in F^{\times}
\right\} \cong GL_1.
\]
We will pass between $SO_2$ and $GL_1$ when convenient.

We now define $f_s \in V_{Q_1}^{SO_{2}}(\mathcal{W}(\tau, \psi^{-1}),s)$ (see \eqref{sectiondefinition}) by
\begin{equation}\label{sectionofprincipalseries}
f_s(h,a) = |\mathrm{det}(h)|^{s-1/2} \tau(ah),
\end{equation}
where $\tau(h)$ and $\mathrm{det}(h)$, mean $\tau$ and $\mathrm{det}$, applied to the upper left entry of $h$.
The map $a \mapsto f(g,a)$ is easily seen to be in the Whittaker model $\mathcal{W}(\tau, \psi^{-1})$ of $\tau$.

Recall the Whittaker function $W$ that we defined earlier (see \eqref{whittakerfunction}) for $\pi_{\chi}^{\zeta} = \mathrm{Ind}_{\langle g_{\chi} \rangle I^+}^{SO_{2\ell+1}} \chi_{\zeta}$.

\begin{lem}\label{supportofW}
If $W(\overline{x} j_{n,\ell}(h)) \neq 0$, then $\overline{x} \in I^{++}$ and $h \in 1 + \mathfrak{p}$.
\end{lem}

\begin{proof}
Let $h = diag(a,a^{-1}) \in SO_2$.  Note that $\overline{x} j_{n,\ell}(h)$ embeds in $SO_{2\ell+1}$ as

\[
\begin{pmatrix}
a \\
ay & I_{\ell-1}\\
& & 1\\
& & & I_{\ell-1}\\
& & & y' & a^{-1}
\end{pmatrix}.
\]

By comparing lower left matrix entries, one can see that $\overline{x} j_{1,\ell}(h) \notin U_{SO_{2\ell+1}}  g_{\chi}  I^+$.  Moreover, by comparing upper left matrix entries, one can see that if $\overline{x} j_{1,\ell}(h) \in U_{SO_{2\ell+1}} I^+$, we must have that $a \in 1 + \mathfrak{p}$.  This implies immediately, by comparing first columns, that every entry in the column vector $y$ is contained in $\mathfrak{p}$.
\end{proof}

Since $\tau$ is tamely ramified, we obtain

\begin{cor}\label{easysidecomputation}
\[
\int_{U_{SO_{2}} \setminus SO_{2}} \int_{\overline{X}_{(1,\ell)}}
W(\overline{x} j_{1,\ell}(h)) f_s(h, 1) d \overline{x} dh = vol(\mathfrak{p})^{\ell-1} vol(1 + \mathfrak{p}).
\]
\end{cor}

\proof
By Lemma \ref{supportofW} and since $U_{SO_2} = 1$, we obtain
\[
\int_{U_{SO_{2}} \setminus SO_{2}} \int_{\overline{X}_{(1,\ell)}}
W(\overline{x} j_{1,\ell}(h)) f_s(h, 1) d \overline{x} dh
\]
\[
=\int_{1 + \mathfrak{p}} \int_{\overline{X}_{(1,\ell)}}
W(\overline{x} j_{1,\ell}(h)) f_s(h, 1) d \overline{x} dh.
\]
Since $\tau$ is tamely ramified, $f_s(h,1) = 1 \ \forall h \in 1 + \mathfrak{p}$.  Applying Lemma \ref{supportofW} again, we have that
\[
\int_{1 + \mathfrak{p}} \int_{\overline{X}_{(1,\ell)}}
W(\overline{x} j_{1,\ell}(h)) d \overline{x} dh
\]
\[
=
\int_{1 + \mathfrak{p}} \int_{\overline{X}_{(1,\ell)} \cap I^{++}}
W(\overline{x} j_{1,\ell}(h)) d \overline{x} dh.
\]
Since $W|_{I^{++}} \equiv 1$, this last double integral equals $vol(\mathfrak{p})^{\ell-1} vol(1 + \mathfrak{p})$, as claimed.
\qed

\section{The computation of $\Phi^*(W, f_s)$}\label{hardside}

In this section, we compute $\Phi^*(W,f_s)$, for the same choice of $W$ and $f_s$ as in \S\ref{easyside}.  Since $\tau$ is a character of $GL_1$, $\Phi^*(W,f_s)$ simplifies to
\begin{align*}
\begin{split}
 & \ \Phi^*(W, f_s)\\
 = & \ \gamma(2s-1, \tau, \wedge^2, \psi) \int_{U_{SO_{2}} \setminus SO_{2}} \int_{\overline{X}_{(1,\ell)}}
W(\hat{c}_{1,\ell} \overline{x} j_{1,\ell}(h) \delta_o \omega') M(\tau,s) f_s(h^{-1}, 1) d \overline{x} dh.
\end{split}
\end{align*}
We note in particular that
\[
\omega' =
\begin{pmatrix}
& & 1\\
& I_{2\ell-1} &\\
1 & &
\end{pmatrix}.
\]

\begin{lem}\label{whittakersupport}
In order that $W(\hat{c}_{1,\ell} \overline{x} j_{1,\ell}(h) \delta_o \omega') \neq 0$, it must be that $h^{-1} \in \varpi \cdot (1 + \mathfrak{p})$, and that $\overline{x} \in I^{++}$.
\end{lem}

\begin{proof}
We first note that if
\[
h =
\begin{pmatrix}
a & 0 \\
0 & a^{-1}
\end{pmatrix},
\]
then
\[
\hat{c}_{1,\ell} \overline{x} j_{1,\ell}(h) \delta_o \omega' =
\begin{pmatrix}
0 & 0 & 0 & 0 & a\\
0 & -I_{\ell-1} & 0 & 0 & -ay\\
0 & 0 & -1 & 0 & 0\\
0 & 0 & 0 & -I_{\ell-1} & 0\\
a^{-1} & 0 & 0 & y' & 0
\end{pmatrix}.
\]
Here (as in \S\ref{easyside}), $y \in F^{\ell-1}$ is a column vector, and $y' \in F^{\ell-1}$ is a row vector.

We now perform some matrix computations.  If $u = (u)_{ij} \in U_{SO_{2\ell+1}}$ and $k = (k)_{ij} \in I^+$, we get that

\[ u g_{\chi} k =
\begin{pmatrix}
& & & & \\
& & * & &\\
\varpi k_{11} & \varpi k_{12} & \cdots & \varpi k_{1n}
\end{pmatrix}
\]

In order that $W(\hat{c}_{1,\ell} \overline{x} j_{1,\ell}(h) \delta_o \omega') \neq 0$, we must have that $\hat{c}_{1,\ell} \overline{x} j_{1,\ell}(h) \delta_o \omega' \in U_{SO_{2\ell+1}} \langle g_{\chi} \rangle I^+$, by definition of $W$.

We first consider the double coset $U_{SO_{2\ell+1}} g_{\chi} I^+$.  Let us compare the last row of $\hat{c}_{1,\ell} \overline{x} j_{1,\ell}(h) \delta_o \omega'$ to the last row of $u g_{\chi} k$.  In particular, since $k_{11} \in 1 + \mathfrak{p}$, we must have that $a^{-1} \in  \varpi \cdot (1 + \mathfrak{p})$.  Since $k_{1j} \in \mathfrak{o}$ for $j \geq 2$, we moreover require that every entry of $y'$ (hence $y$) is in $\mathfrak{p}$.

A similar type of computation shows that $\hat{c}_{1,\ell} \overline{x} j_{1,\ell}(h) \delta_o \omega' \notin U_{SO_{2\ell+1}} I^+$.
\end{proof}
We now have our main result.
\begin{thm}\label{gammafactor}
\[
\gamma(s, \pi \times \tau, \psi) = \chi_{\zeta}(g_{\chi}) \tau(-\varpi) q^{1/2 - s}.
\]
\end{thm}

\proof
We first note that since $\tau$ is a character of $GL_1$, we have that $\gamma(2s-1, \tau, \wedge^2, \psi) = 1$.

Now let $h = diag(a,a^{-1}) \in SO_2$, and suppose that $W(\hat{c}_{1,\ell} \overline{x} j_{1,\ell}(h) \delta_o \omega') \neq 0$. By Lemma \ref{whittakersupport}, we have that $h \in 1 + \mathfrak{p}$ and $a^{-1} \in  \varpi \cdot (1 + \mathfrak{p})$.  In particular, we can write $a = \varpi^{-1} v$, for some $v \in 1 + \mathfrak{p}$.

We now decompose $\hat{c}_{1,\ell} \overline{x} j_{1,\ell}(h) \delta_o \omega'$ in the double coset $U_{SO_{2\ell+1}} g_{\chi} I^+$ as

\[
\hat{c}_{1,\ell} \overline{x} j_{1,\ell}(h) \delta_o \omega' =
\begin{pmatrix}
0 & 0 & 0 & 0 & a\\
0 & -I_{\ell-1} & 0 & 0 & -ay\\
0 & 0 & -1 & 0 & 0\\
0 & 0 & 0 & -I_{\ell-1} & 0\\
a^{-1} & 0 & 0 & y' & 0
\end{pmatrix}.
\]
\[
=
\begin{pmatrix}
 & & & & \varpi^{-1}\\
& & & & \\
&  & -I_{2 \ell-1} & & \\
& & & & \\
\varpi & & & &
\end{pmatrix}
\begin{pmatrix}
v^{-1} & & & \varpi^{-1} y' &  \\
& I_{\ell-1} & & & ay\\
& & 1 & & \\
& & & I_{\ell-1} & \\
& & & & v
\end{pmatrix}.
\]
This decomposition implies by Lemma \ref{whittakersupport} that $W(\hat{c}_{1,\ell} \overline{x} j_{1,\ell}(h) \delta_o \omega') = \chi_{\zeta}(g_{\chi})$.  Finally, one can see that $M(\tau,s) f_s(h^{-1}, 1) = f_s(h^{-1}, 1)$.  Since the support of $W$ forced $h \in 1 + \mathfrak{p}$, we have that $f_s(h^{-1},1) = 1$.  Therefore,

\begin{align*}
\begin{split}
 & \ \Phi^*(W, f_s)\\
= & \ \int_{SO_{2}} \int_{\overline{X}_{(1,\ell)}}
W(\hat{c}_{1,\ell} \overline{x} j_{1,\ell}(h) \delta_o \omega') M(\tau,s) f_s(h^{-1}, 1) d \overline{x} dh\\
 = & \ vol(\mathfrak{p})^{\ell-1} \ \int_{1 + \mathfrak{p}}
\chi_{\zeta}(g_{\chi})  |\mathrm{det}(-\varpi v^{-1})|^{s-1/2} \tau(-\varpi v^{-1}) dh \\
= & \ vol(\mathfrak{p})^{\ell-1} \chi_{\zeta}(g_{\chi}) |\mathrm{det}(\varpi )|^{s-1/2} \tau(-\varpi ) vol(1 + \mathfrak{p})\\
= & \ vol(\mathfrak{p})^{\ell-1} \chi_{\zeta}(g_{\chi}) q^{1/2-s} \tau(-\varpi ) vol(1 + \mathfrak{p}),
\end{split}
\end{align*}
since $\tau$ is tamely ramified.  Combining this computation with Corollary \ref{easysidecomputation} and \eqref{functionalequation}, we have our result.
\qed

We now compare the $\gamma(s, \pi_{\chi}^{\zeta} \times \tau, \psi)$ with the gamma factors of simple supercuspidal representations of $GL_{2\ell}$.

\begin{cor}\label{funclift}
The representations $\pi_{\chi}^{\zeta}$ of $SO_{2 \ell + 1}$ and $\sigma_{\chi}^{\zeta}$ of $GL_{2\ell}$ (that has trivial central character) share the same twisted gamma factors, where the twisting is over all tamely ramified characters $\tau$ of $GL_1$.  That is,
\[
\gamma(s, \pi_{\chi}^{\zeta} \times \tau, \psi) = \gamma(s, \sigma_{\chi}^{\zeta} \times \tau, \psi),
\]
for all tamely ramified characters $\tau$ of $GL_1$.
\end{cor}

\proof
This is just a matching of the gamma factors from Theorem \ref{gammafactor} and Lemma \ref{simplesupercuspidalgammagln}.
\qed

\section{Depth-preservation}\label{depthpreservation}
In this section, we briefly recall a conjecture about preservation of depth in the local Langlands correpsondence.  We then show that if one assumes the conjecture, the representation $Ind_{W_E}^{W_F} \xi$ of $W_F$ that we constructed in \S\ref{introduction} is in fact the Langlands parameter of $\pi_{\chi}^{\zeta}$.

We recall that the depth of a representation $\varphi : W_F' \rightarrow {}^L G$ of the Weil-Deligne group is defined as $\mathrm{inf} \{r \ | \ \varphi(I_s) = 1 \ \mathrm{for \ all} \ s > r \}$, where $I_s, s \in \mathbb{Q}$, is the filtration of the inertia subgroup $I$ of $W_F$.

Suppose that $\mathbf{G}$ is a tamely ramified $p$-adic group, and let $\varphi : W_F' \rightarrow {}^L G$ be the Langlands parameter of a representation $\pi$ of $\mathbf{G}(F)$.  It is conjectured that if $p$ is sufficiently large, then $depth(\pi) = depth(\varphi)$.

\begin{cor}\label{functoriallift}
Assume that the depth perservation conjecture holds.  Then, for $p$ sufficiently large, the Langlands parameter of $\pi = \pi_{\chi}^{\zeta}$ is $Ind_{W_E}^{W_F} \xi$.
\end{cor}

\proof
Suppose, to the contrary, that the Langlands parameter $\varphi_{\pi} : W_F \rightarrow Sp(2\ell,\mathbb{C}) \hookrightarrow GL(2\ell,\mathbb{C})$ is reducible.  Write $\varphi_{\pi} = \displaystyle\bigoplus_{i=1}^m \varphi_i$, with $\varphi_i : W_F \rightarrow GL(n_i,\mathbb{C})$.  Since $depth(\varphi_i) \in \{ \frac{a}{n_i} \ | \ a \in \mathbb{N} \} \cup \{0 \}$, and since $depth(\varphi) = \frac{1}{2\ell}$ by assumption, we see immediately that $m = 1$.

We therefore conclude, in particular, that the functorial lift of $\pi$ to $GL(2\ell,F)$ (under the embedding $Sp(2\ell,\mathbb{C}) \hookrightarrow GL(2\ell,\mathbb{C})$) is supercuspidal and has depth $\frac{1}{2\ell}$, so in fact is simple supercuspidal.  We also note that this lift must have trivial central character since it is a lift from $SO_{2\ell+1}$.  By Corollary \ref{funclift} and by \cite[Remark 3.18]{AL14}, this lift must be $\sigma_{\chi}^{\zeta}$ (with trivial central character).  The main result of \cite[\S3]{AL14} says that the Langlands parameter of $\sigma_{\chi}^{\zeta}$ is $Ind_{W_E}^{W_F} \xi$.
\qed

\end{document}